\documentclass[a4paper,12pt]{article}
\usepackage{amscd,amsthm, amsmath, amsfonts,amssymb,epsfig} 
\usepackage{amsfonts}
\usepackage{color,graphics}

\makeatletter

\@addtoreset{equation}{section}

\newtheorem{theorem}{Theorem}[section]
\newtheorem{lemma}[theorem]{Lemma}

\newtheorem{corollary}[theorem]{Corollary}
\theoremstyle{definition}
\newtheorem{de}[theorem]{Definition}

\newtheorem{ex}[theorem]{Example}
\newtheorem{re}[theorem]{Remark}
\newtheorem{konst}[theorem]{Construction}

\newtheorem{definition}{Definition}

\makeatother

\def\e{\begin{equation}}
\def\ee{\end{equation}}
\def\Q{\mathbb Q}
\def\1{\frac{1}{2}}

\def\oMn{\overline{{\mathcal M}_{0,n}^{\mathbb R}}}

\def\R{{\mathbb R}}

\def\Th{W_{n-4}}
\def\Thh{W_1}
\def\ep{\varepsilon}

\begin{document}

\title{Computation of the first  Stiefel-Whitney class for the variety  $\overline{{\mathcal M}_{0,n}^{\mathbb R}}$\thanks{The work of the first author is partially supported by the grants NSh-1500.2014.2, RFBR 13-02-00478. The work of the second author is partially supported by the grant RFBR 12-01-00140}}
\author{N. Ya. Amburg, E. M. Kreines}
\date{}

\maketitle
\vspace{-6cm}
\begin{center}
\hfill ITEP-TH-33/14\\
\end{center}
\vspace{5cm}

\abstract
We compute the class $\Th(\overline{{\mathcal M}_{0,n}^{\mathbb R}})$ which is Poincar\'e dual to the first  Stiefel-Whitney  class for the variety  $\overline{{\mathcal M}_{0,n}^{\mathbb R}}$ in terms of the natural cell decomposition of $\overline{{\mathcal M}_{0,n}^{\mathbb R}}$ .

\section{Introduction}

Let $\overline{{\mathcal M}_{0,n}^{{\mathbb R}}}$  be the Deligne-Mumford compactification of the moduli space
of algebraic curves of genus 0 with $n$ marked and numbered  points.  In  \cite{EtingofCo,Dev,Kap} topology of this variety is studied. Many topological and algebraic characteristics of this variety are investigated. In particular, the structure of the cell decomposition is introduced. By means of this cell decomposition Euler characteristics of the variety is computed and Betti numbers are found.

This paper is devoted to the investigations of the first Stiefel-Whitney  class for the variety  $\overline{{\mathcal M}_{0,n}^{\mathbb R}}$. More precisely, we provide a natural geometric interpretation of the class
 $\Th(\overline{{\mathcal M}_{0,n}^{{\mathbb R}}})$, which is Poincar\'e dual to the first Stiefel-Whitney  class of $\overline{{\mathcal M}_{0,n}^{{\mathbb R}}}$ in terms of the natural cell decomposition of  $\overline{{\mathcal M}_{0,n}^{\mathbb R}}$.

The main result of the paper is the proof of the following two statements.
\begin{theorem} \label{them1}
Let $n\ge 5$, ${\mathcal M}_{0,n}^\R$  be the real moduli space of algebraic curves of the genus 0 with $n$ marked and numbered   points, say $\{1,2,\ldots, n\}$. Let $\overline{{\mathcal M}_{0,n}^\R}$ be its   Deligne-Mumford compactification. We consider the cell decomposition of $\overline{{\mathcal M}_{0,n}^\R}$, defined in the Construction \ref{l:oMn_cell_dec}. Then the class  $\Th(\overline{{\mathcal M}_{0,n}^\R})$ (which is Poincar\'e dual to the first Stiefel-Whitney class of  $\overline{{\mathcal M}_{0,n}^\R}$) consists exactly from those cells of co-dimension 1, that satisfy the condition: irreducible component of the curve which contains at most one point from the set $\{1, 2,3\}$ contains an odd number of points from the set $\{1,2,\ldots, n\}$.
\end{theorem}

\begin{corollary}
Let $n\ge 6$ be even and let $\overline{{\mathcal M}_{0,n}^{\R}}$ be the Deligne-Mumford compactification of the moduli space
of algebraic curves of genus 0 with $n$ marked and numbered points. We consider the cell decomposition of $\overline{{\mathcal M}_{0,n}^\R}$, defined in the Construction \ref{l:oMn_cell_dec}. Then the class  $\Th(\overline{{\mathcal M}_{0,n}^\R})$ consists exactly from the cells of co-dimension 1, such that each  irreducible component of the curve  contains an odd number of the marked points. 
\end{corollary}

The first Stiefel-Whitney class of this variety was firstly computed in~\cite{Ceyhan}. In this work we point out another representative for the class which is Poincar\'e dual to the  first Stiefel-Whitney class of  $\overline{{\mathcal M}_{0,n}^{\mathbb R}}$. The advantage of our method is the direct application of the corresponding cell decomposition structure. We provide easy combinatorial and geometric characterization of cells which provides the possibility to determine if the cell lies in the class under consideration or not.

Further we plan to apply the obtained characterization for geometric interpretation of generators and relations in the cohomology algebra $H^*(\overline{{\mathcal M}_{0,n}^\R},\Q)$. We are going to prove their close connections with  those that where find by Keel  \cite{Keel} in the complex case and to point out the principal difference with the complex case.

Our work is organized as follows: Section 2 contains necessary definitions and detailed description of the cell decomposition of the variety   $\overline{{\mathcal M}_{0,n}^{\mathbb R}}$. This theory is illustrated by the description of the cell decomposition of  $\overline{{\mathcal M}_{0,5}^{\mathbb R}}$, which includes the number of cells, adjacency types, graphic illustration of the adjacency. Section 3 is devoted to the computation of the  first Stiefel-Whitney class separately for  $n=5$ and $n\ge 6$.

\section{Deligne-Mumford compactification of the space  ${\mathcal M}_{0,n}^{\mathbb R}$} 
\subsection{Stable curves}

Following the paper \cite{EtingofCo}, we define real stable curves as ``cacti-like'' structures, i.e. 3-dimensional  ``trees'' (in the graph theoretical sense), consisting of flat circles with the points $\{1,2,\ldots,n\}$ on them:
 
\begin{definition} \label{SC} \cite{DM}

A {\em stable curve\/} of genus 0 with $n$ marked points over the field of real numbers  $\R$ is a finite union of real projective lines 
$C=C_1\cup C_2\cup \ldots\cup C_p$  with $n$ different marked points  $z_1,z_2,\ldots,z_n\in C$, if the following conditions hold.
\begin{enumerate}
\item For each point  $z_i$ there exists the unique line  $C_j$, such that $z_i\in C_j$.
\item For any pair of lines $C_i\cap C_j$ is either empty or consists of one point, and in the latter
case the intersection is transversal.
\item \label{it3} The graph corresponding to  $C$ (the lines $C_1,C_2,\ldots,C_p$ correspond to the vertices; two vertices are incident to the same edge iff corresponding lines have non-empty intersection) is a tree.   
\item  \label{it4} The total number of special points (i.e. marked points or intersection
points) that belong to a given component $C_j$ is at least 3 for each $j=1,\ldots,p$.
\end{enumerate}
We say that  $p$ is the  {\em number of components\/} of the stable curve.
\end{definition}
\begin{figure}[h]
\centering
\epsfig{figure=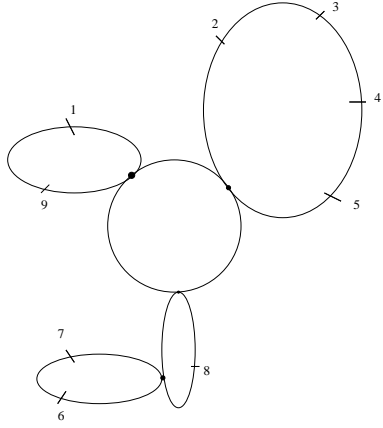,width=.30\linewidth}
\caption{A stable curve over  $\mathbb R$ of genus  $0$ with $9$ marked points}
\label{exstcurve1}
\end{figure}
\begin{definition}
Let  $C=(C_1,C_2,\ldots,C_p,z_1,z_2,\ldots,z_n)$ and $C'=(C'_1,C'_2,\ldots,C'_p,z'_1,z'_2,\ldots,z'_n)$ be stable curves of genus 0 with $n$ marked and numbered points. $C$ and $C'$ are called {\em equivalent\/} if there exists an isomorphism of algebraic curves $f:C\to C'$ such that   $f(z_i)=z'_i$ for all $i=1,\ldots , n$.
\end{definition}

\subsection{Moduli space $\overline{{\mathcal M}_{0,n}^{\mathbb R}}$}
\begin{definition}
Let  $n\ge 3$. {\em Deligne-Mumford compactification\/}  of the moduli space of genus 0 real algebraic curves with $n$ marked points 
$\overline{{\mathcal M}_{0,n}^{\mathbb R}}$  is the set of equivalence classes of the genus 0 stable curves with $n$ marked points defined over ${\mathbb R}$.
\end{definition}
\begin{theorem}\cite{Dev}
Let $n>4$. Then the space $\oMn$ is a non-orientable compact variety of real dimension  $\dim(\oMn)=n-3$.
\end{theorem}

\subsection{Cell decomposition of  $\overline{{\mathcal M}_{0,n}^{\mathbb R}}$}
\begin{re}
There exists a natural structure of  cell decomposition for the space $\overline{{\mathcal M}_{0,n}^{\mathbb R}}$. This structure is described for example in the works \cite{Dev,Kap}. 
\end{re}
\begin{konst} \cite{Dev,Kap}  \label{l:oMn_cell_dec}
{\bf Description of the cell decomposition of the compactified space $\overline{{\mathcal M}_{0,n}^{\mathbb R}}$. }
Following \cite{Dev}, let us consider a right $n$-gon, possibly, with several non-intersecting diagonals. We label different cells of the moduli space $\overline{{\mathcal M}_{0,n}^{\mathbb R}}$  by such $n$-gons, sides  of which correspond to the marked points and are labeled by  $z_1,\ldots, z_n$. Here, the polygons which can be transformed to each other by the dihedral group action label the same cell. The cells of the maximal dimension are labeled by $n$-gons without diagonals. The cells of codimension 1 are labeled by $n$-gons with one diagonal. Note that these cells consist exactly of 2-component stable curves. The cells of codimension 3, i.e., that correspond to 3-component curves, are labeled by  $n$-gons with 2 diagonals. In general, a cell of codimension $k$ is labeled by an  $n$-gon $M$ with $k$ diagonals. These diagonals divide $M$ into $k+1$ polygons $M_1,\ldots, M_{k+1}$. The edges of $M_1,\ldots, M_{k+1}$, which are the edges of $M$, are labeled by the points marking the components of the curve. Note that the condition  \ref{it3} guarantees that different diagonals do not intersect outside the vertices of  $M$. The condition \ref{it4} guarantees that each of the polygons  $M_1,\ldots, M_{k+1}$ has at least 3 sides, i.e., it is a polygon.

Following \cite{Dev}, we denote by ${\mathcal G}^L(n,k)$ the set of $n$-gons with $k$ non-intersecting diagonals and labels  $z_1,\ldots, z_n$ on the edges.
\end{konst}

\begin{de}
A {\em twist\/} of $M \in {\mathcal G}^L(n,k)$ along a diagonal $d$ is the $n$-gon $M' \in {\mathcal G}^L(n,k)$, obtained from $M$ by cutting $M$ along  $d$,  then $180^\circ$ rotation  of one (any one) of the parts relative to the axis orthogonal to $d$ in the plane of the $n$-gon, and gluing the obtained two parts along~$d$. 
\end{de}
\begin{re}
Let $M'$ be the twist of $M$, let labels of the sides of  $M$ be ordered as  $z_1,\ldots,z_n$, and let the sides marked by $z_1,\ldots,z_k$ be separated by $d$ from the sides marked by $z_{k+1},\ldots, z_n$. Then the sides of $M'$ have ordered labels $z_1,\ldots,z_k,z_{n}, z_{n-1},\ldots, z_{k+1}$.  
\end{re}

\begin{konst} \cite{Dev,Kap}  \label{l:oMn_cell_dec1}
{\bf  Description of the cell decomposition of the compactified space $\overline{{\mathcal M}_{0,n}^{\mathbb R}}$. Prolongation. }
Let polygons $M$ and $M'\in {\mathcal G}^L(n,k)$ can be transformed to each other by the series of twists. Then $M$ and $M'$ mark the same cell. 
\end{konst}
\begin{re}
Special charm of this construction is that marked points and singular points (the points of intersection of different components of a curve) do not have principal differences. Namely, both of them are denoted by edges of a polygon. Also, each component  $C_i$ of the curve (as well as a connected union of several components) is denoted by a polygon. This polygon marks the cell of the cell decomposition (for the moduli space of a smaller dimension), which contain $C_i$.
\end{re}
\begin{ex} \label{ex1}
Cell decomposition of $\overline{{\mathcal M}_{0,n}^{\mathbb R}}$ contains $\frac{(n-1)!}{2}$ cells of the maximal dimension $n-3$.
\end{ex}
\begin{ex}
$\overline{{\mathcal M}_{0,4}^{\mathbb R}}$ is a circle consisting of 3 cells of the dimension  $1$ and 3 cells of the dimension~$0$.
Figure \ref{M_0,4} represents the cell decomposition of $\overline{{\mathcal M}_{0,4}^{\mathbb R}}$. Nearby each cell we provide its ``typical'' representative, i.e., one of the stable curves which constitute this cell.
\begin{figure}[h]
\centering
\epsfig{figure=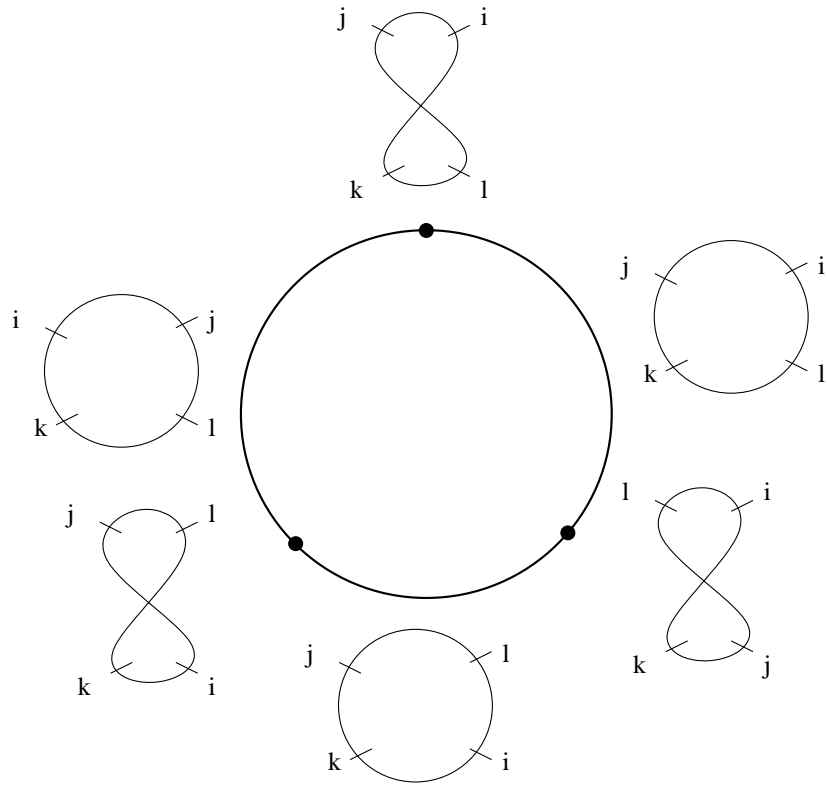,width=.65\linewidth}
\caption{$\overline{{\mathcal M}_{0,4}^{\mathbb R}}$.}
\label{M_0,4}
\end{figure}
\end{ex}

\subsection{Cell decomposition of the variety  $\overline{{\mathcal M}_{0,5}^{\mathbb R}}$}
By Example \ref{ex1} found cell decomposition of $\overline{{\mathcal M}_{0,5}^{\mathbb R}}$ consists of 12 cells of the maximal dimension. Each cell is labeled by a pentagon. The cell  labeled by the pentagon with sides marked by the symbols 1, 2, 3, 4, 5 is represented at Figure \ref{fig5_1}. All curves which are in this cell have the form shown at Figure \ref{fig5_1} inside the pentagon. Outside the pentagon the stable curves are shown that are obtained by moving to each of the boundaries. At the Figure \ref{Klgr0} two cells of codimension 1 are shown. The cell marked by the letter $A$ corresponds to the lower edge of the pentagon drawn  at Figure \ref{fig5_1}. The cell marked by $B$    corresponds to the next edge in the counterclockwise order.

\begin{center}

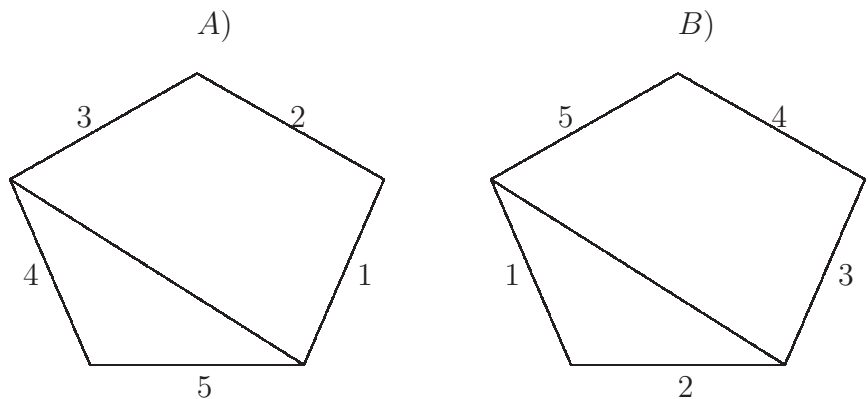
\begin{figure}[h]
\centering
\begin{picture}(260,140)
\put(0,0){\line(1,0){80}}
\put(180,0){\line(1,0){80}}
\qbezier{(-30,70)(-15,35)(0,0)}
\qbezier{(110,70)(95,35)(80,0)}
\qbezier{(-30,70)(5,90)(40,110)}
\qbezier{(40,110)(75,90)(110,70)}
\qbezier{(150,70)(165,35)(180,0)}
\qbezier{(290,70)(275,35)(260,0)}
\qbezier{(150,70)(185,90)(220,110)}
\qbezier{(220,110)(255,90)(290,70)}
\qbezier{(150,70)(205,35)(260,0)}
\qbezier{(-30,70)(25,35)(80,0)}
\put(40,125){$A)$}
\put(220,125){$B)$}
\put(40,-12){$5$}
\put(100,30){$1$}
\put(75,90){$2$}
\put(-5,90){$3$}
\put(-25,30){$4$}
\put(220,-12){$2$}
\put(280,30){$3$}
\put(255,90){$4$}
\put(175,90){$5$}
\put(155,30){$1$}
\end{picture}
\caption{Forms of the boundary cells}
 \label{Klgr0}
\end{figure}

\end{center}

Two adjoint cells are shown at Figure  \ref{fig5_2}.
Adjoining  between all 12 cells is shown at Figure \ref{fig5_3}.
\begin{figure}[h]
\centering
\epsfig{figure=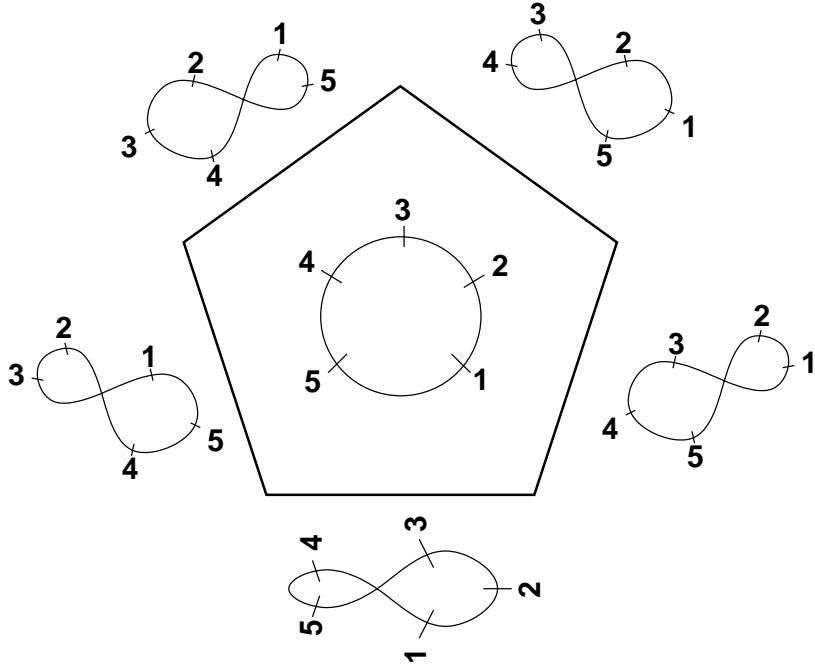,width=.65\linewidth}
\caption{One of cells of $\overline{{\mathcal M}_{0,5}^{\mathbb R}}$.} \label{fig5_1}
\end{figure}
\begin{figure}[t] 
\centering
\epsfig{figure=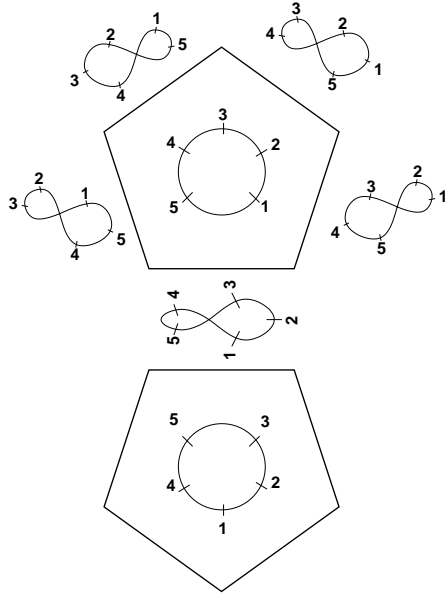,width=.35\linewidth}
\caption{Two adjoint  cells of $\overline{{\mathcal M}_{0,5}^{\mathbb R}}$.} \label{fig5_2}
\end{figure}
\begin{figure}[t] 
\centering\epsfig{figure=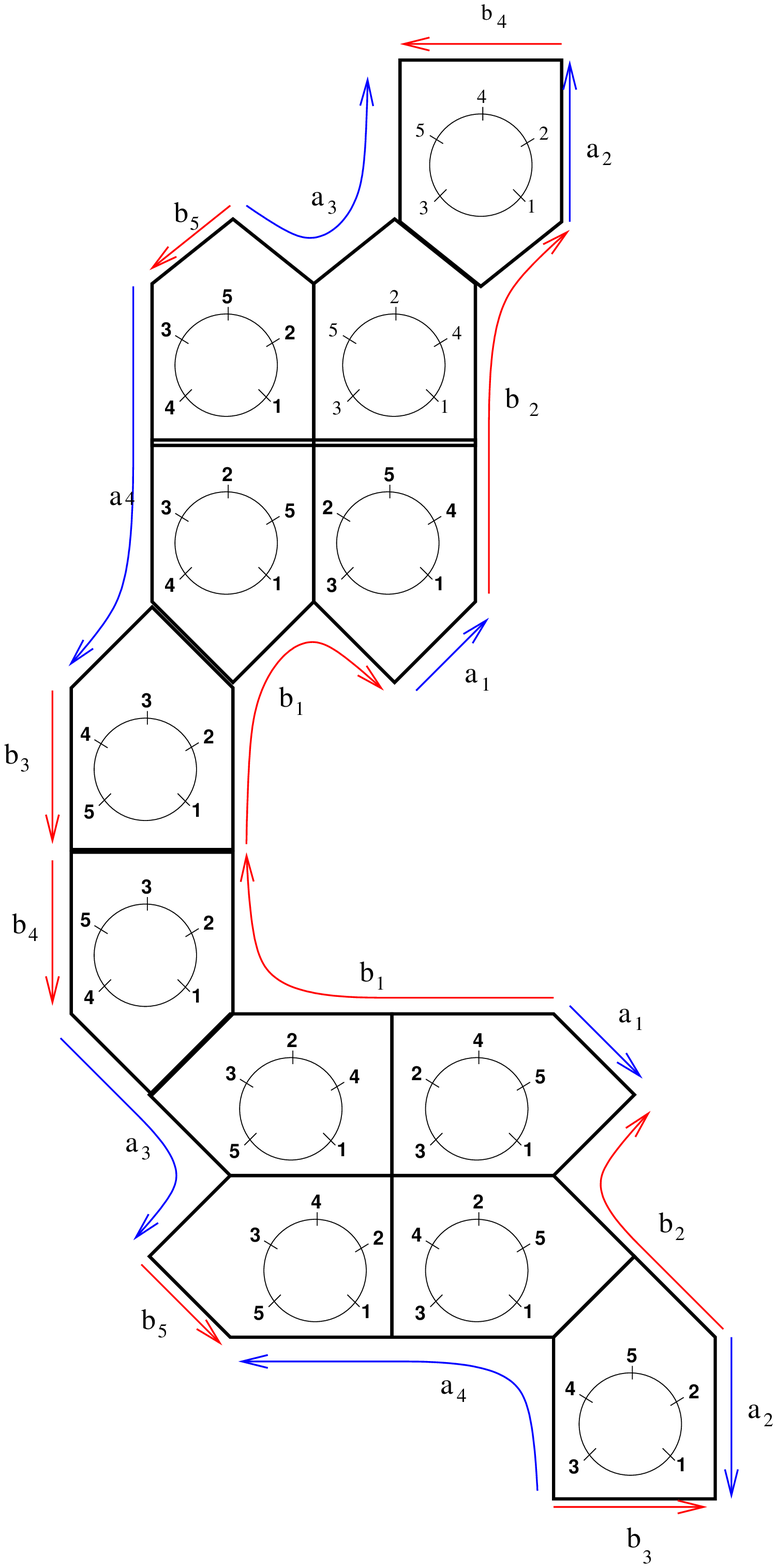,width=.35\linewidth}
\caption{$\overline{{\mathcal M}_{0,5}^{\mathbb R}}$} \label{fig5_3}
\end{figure}
\begin{figure}[t]
\centering\epsfig{figure=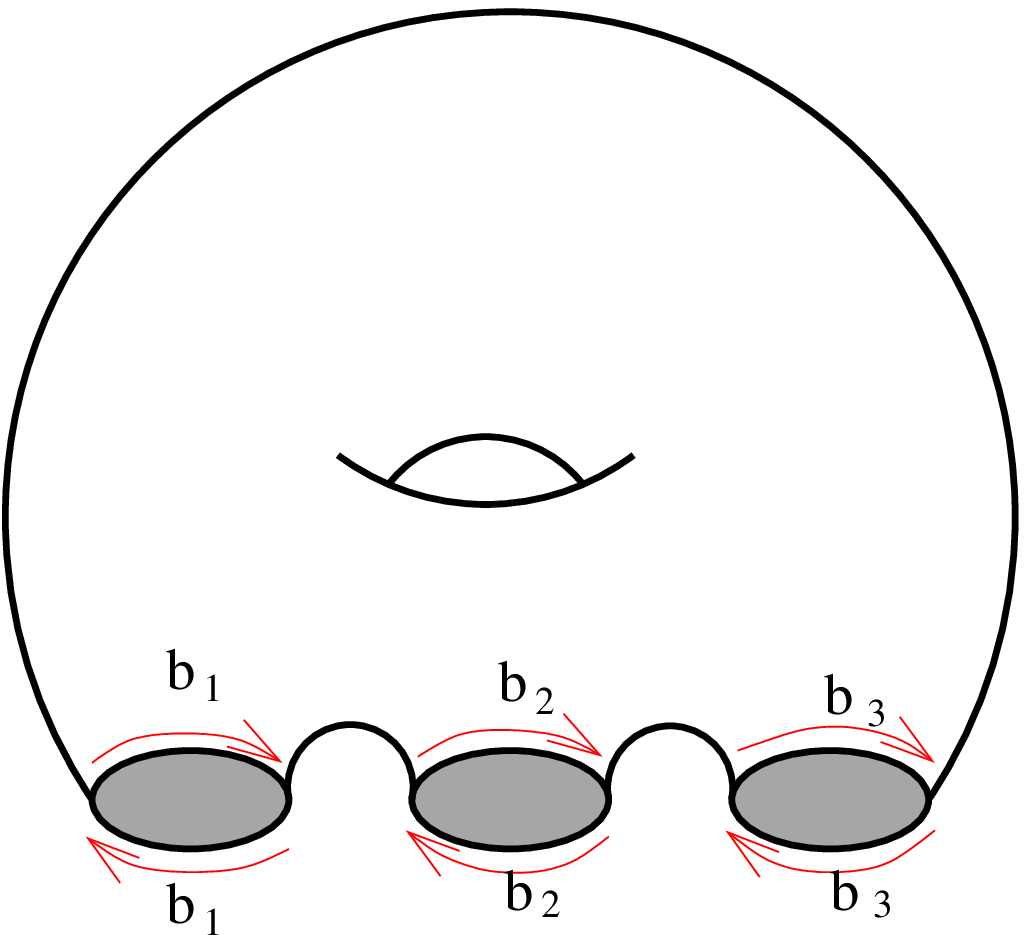,width=.35\linewidth}
\caption{$\overline{{\mathcal M}_{0,5}^{\mathbb R}}$}  \label{fig5_4}
\end{figure}

\section{ Stiefel-Whitney class of $\overline{{\mathcal M}_{0,n}^{\mathbb R}}$}
\subsection{Some general remarks}

\begin{re}
We use the following theorem in order to compute the homological class $\Th$, which is Poincar\'e dual to the first  Stiefel-Whitney class of the variety $\overline{{\mathcal M}_{0,n}^{\mathbb R}} $. For the computations we use introduced structure of the natural cell decomposition of $\overline{{\mathcal M}_{0,n}^{\mathbb R}}$.
\end{re}

\begin{theorem} \label{thmShU}  \cite[p. 119]{Mil},~\cite{HT,Ceyhan} 
Let  $M$ be a smooth compact variety without a boundary, $K$ be a cell decomposition of  $M$, $k_j\subset K$ denote the cells of the maximal dimension $d$. Let us fix the orientation on the cells of maximal dimension $\overline{k_j}$.
Then homological class dual to the first  Stiefel-Whitney class of $M$ can be represented in the form 
\begin{equation} \label{hru} W_{d-1}(M)=\left( \frac{1}{2}\sum \partial \overline{k_j} \right) \mod 2.\end{equation}
\end{theorem}

The main idea of the proof of Theorem \ref{them1} is to introduce global coordinates on the moduli space  ${{\mathcal M}_{0,n}^{\mathbb R}}$ which determine the orientation in each cell. Certainly, these coordinates can be prolonged till the boundary for some cells, but not all cells. If the introduced coordinates can not be prolonged till the boundary, we introduce some other coordinates, which can be prolonged till that boundary, and compute the Jacobian of the transition functions between different coordinates. Then we can determine if two adjoint cells induce the same orientations on their common boundary, or the opposite ones. If two adjoint cells induce the opposite orientations on their common boundary, then their influences in the formula (\ref{hru}) annihilate each other. Otherwise, we add the influences. Then after dividing by 2 and considering the result modulus 2 we get that the cell is in the class $\Th(\overline{{\mathcal M}_{0,n}^{\mathbb R}})$ with the coefficient 1. So, the cell is in $\Th(\overline{{\mathcal M}_{0,n}^{\mathbb R}})$ iff it is a common boundary of two cells of the maximal dimension, which have the same orientations (in the global coordinates determined in the open part of the moduli space). 

We are going to realize this program.

In order to find the first  Stiefel-Whitney class of the variety $\overline{{\mathcal M}_{0,n}^{\mathbb R}} $ we fix one of the possible orientations for the maximal dimension cells of the space ${{\mathcal M}_{0,n}^{\mathbb R}}$.

\begin{de} \label{DefKoord}
A {\it coordinate map} on the space ${{\mathcal M}_{0,n}^{\mathbb R}}$ is  the map $\varphi: {{\mathcal M}_{0,n}^{\mathbb R}}\to {\mathbb R}^{n-3}$.

Let  $({\mathbb P}_1(\R), z_1,\ldots,z_n )\in {{\mathcal M}_{0,n}^{\mathbb R}}$, $z_i\in {\mathbb P}_1(\R)$, i.e., we consider 1-component curve, which is a projective line, with $n$ marked points on it. Since the points $ z_1,\ldots,z_n \in {\mathbb P}_1(\R)$ are determined up to the linear-fractional transformation, we can and we do fix $z_1=0,z_2=1,z_3=\infty$.We define
$$\varphi({\mathbb P}_1(\R), z_1,\ldots,z_n )
=(z_4,\ldots,z_n).$$ 
The chosen system of coordinates $(z_4,\ldots,z_n)$ in a cell or several cells of the space  ${\mathcal M}_{0,n}^{\mathbb R}$ we call the {\em parametrization \/} of these cells.
\end{de}

We fix the standard orientation of the space $\R^{n-3}$. This determines the orientation on the cells of the maximal dimension. In the chosen parametrization we can easily see the boundaries of the cells shown at Figure \ref{Klgr2}, which correspond to the gluing of the points marked by  $i$ and $j$, where $1\le i \le n$, $4\le j\le n$, i.e., the boundaries drawn at Figure \ref{Klgr1}. 

\begin{center}

\begin{figure}[h]
\centering
\begin{picture}(160,90) 
\setlength{\unitlength}{0.5pt}
\qbezier{(-30,140)(-40,105)(-50,70)}
\qbezier{(110,140)(120,105)(130,70)}
\qbezier{(-30,140)(5,160)(40,180)}
\qbezier{(40,180)(75,160)(110,140)}
\qbezier{(-20,20)(-35,45)(-50,70)}
\qbezier{(100,20)(115,45)(130,70)}
\put(0,5){$\bullet$}
\put(20,0){$\bullet$}
\put(40,-5){$\bullet$}
\put(60,0){$\bullet$}
\put(80,5){$\bullet$}
\put(75,163){$j$}
\put(-5,163){$i$}
\put(260,80){$1\le i \le n$,}
\put(260,60){$4\le j\le n$.}
\put(160,70){, \   where}
\end{picture}
\caption{A cell of the maximal dimension}
 \label{Klgr2}
\end{figure}
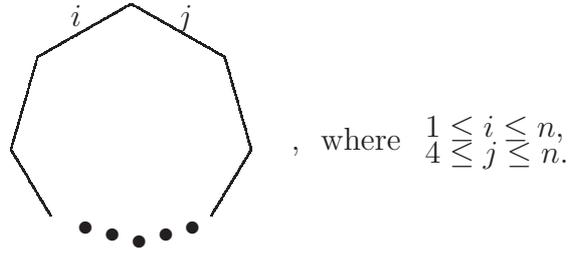
\end{center}


\begin{center}
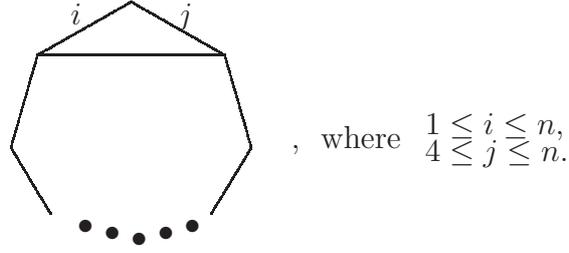
\begin{figure}[h]
\centering
\begin{picture}(160,90) 
\setlength{\unitlength}{0.5pt}
\qbezier{(-30,140)(-40,105)(-50,70)}
\qbezier{(110,140)(120,105)(130,70)}
\qbezier{(-30,140)(5,160)(40,180)}
\qbezier{(40,180)(75,160)(110,140)}
\qbezier{(-30,140)(70,140)(110,140)}
\qbezier{(-20,20)(-35,45)(-50,70)}
\qbezier{(100,20)(115,45)(130,70)}
\put(0,5){$\bullet$}
\put(20,0){$\bullet$}
\put(40,-5){$\bullet$}
\put(60,0){$\bullet$}
\put(80,5){$\bullet$}
\put(75,163){$j$}
\put(-5,163){$i$}
\put(260,80){$1\le i \le n$,}
\put(260,60){$4\le j\le n$.}
\put(160,70){, \ where}
\end{picture}
\caption{Cells $K_{ij|l_1\ldots l_{n-2}}$}
\label{Klgr1}
\end{figure}
\end{center}

\begin{lemma} \label{reShU} For all $n\ge 5$ and for all $i,j$, $1\le i\le n,\ 4\le j\le n$ the cells $K_{ij|l_1\ldots l_{n-2}}$, drawn at Figure  \ref{Klgr1},  are not in the class $\Th(\overline{{\mathcal M}_{0,n}^{\mathbb R}})$.
\end{lemma}

\begin{proof}
For all $i,j$,  $1\le i\le n,\ 4\le j\le n$, the cells of the maximal dimension having the cell at  Figure  \ref{Klgr1} as their common bound induce the opposite orientations on this cell. Hence, by Theorem \ref{thmShU} the common boundary $K_{ij|l_1\ldots l_{n-2}}$ is included to the sum  (\ref{hru})  twice with the opposite signs, hence, it is not in the class $\Th(\overline{{\mathcal M}_{0,n}^{\mathbb R}})$.
\end{proof}

\subsection{Computation of $\Thh(\overline{{\mathcal M}_{0,5}^\R})$}

By Lemma \ref{reShU} it remains to consider the cells of codimension 1 of the form

\begin{center}
\begin{picture}(100,140)
\put(0,0){\line(1,0){80}}
\qbezier{(-30,70)(-15,35)(0,0)}
\qbezier{(110,70)(95,35)(80,0)}
\qbezier{(-30,70)(5,90)(40,110)}
\qbezier{(40,110)(75,90)(110,70)}
\qbezier{(-30,70)(25,70)(110,70)}
\put(75,90){$j$}
\put(-5,90){$i$}
\end{picture}
\end{center}

where $1 \le i,j \le 3$.

We start with $i=1$, $j=2$.

\begin{lemma} \label{K12345} The boundary cell labeled by the pentagon

\begin{center}
\begin{picture}(100,140)
\put(0,0){\line(1,0){80}}
\qbezier{(-30,70)(-15,35)(0,0)}
\qbezier{(110,70)(95,35)(80,0)}
\qbezier{(-30,70)(5,90)(40,110)}
\qbezier{(40,110)(75,90)(110,70)}
\qbezier{(-30,70)(25,70)(110,70)}
\put(40,-12){$4$}
\put(100,30){$3$}
\put(75,90){$2$}
\put(-5,90){$1$}
\put(-25,30){$5$}
\put(-15,120){$K_{12|345}$}
\end{picture}
\end{center}
is in the class  $\Thh(\overline{{\mathcal M}_{0,5}^{\mathbb R}})$.
\end{lemma}

\begin{proof}
1. Let us consider the coordinates on ${{\mathcal M}_{0,5}^{\mathbb R}}$ which can be prolonged to this boundary. To do this, similarly with Definition \ref{DefKoord} we construct an appropriate coordinate map $\varpi_1: {{\mathcal M}_{0,5}^{\mathbb R}} \to {\mathbb R}^2$. Let $({\mathbb P}_1(\R),y_1,\ldots, y_5)\in {{\mathcal M}_{0,5}^{\mathbb R}}$ be a point of the moduli space ${\mathcal M}_{0,5}^\R$. Let us consider a parametrization of the curve ${\mathbb P}_1(\R) $ such that $y_3 =\infty, y_4=0, y_5=1$. We set $\varphi_1({\mathbb P}_1(\R),y_1,\ldots,y_5)=(y_1,y_2)$. 

We now find the transition function from the coordinates $(z_4,z_5)$ to the coordinates $(y_1,y_2)$. To do this, we write $z$- and $y$-coordinates on  ${\mathbb P}_1(\R) $ for each of the 5 marked points and find the rational-fractional transformation, which maps the coordinates to each other. The condition of the coordinate changing has the form
$$
\begin{array}{cccccc} i& 1&2&3&4&5 \\ z-\mbox{coordinates } &0&1&\infty& z_4&z_5 \\y-\mbox{coordinates } &y_1&y_2&\infty&0&1 \end{array}$$

We are looking for the linear-fractional transformation of the form  
$f(t)=\frac{at+b}{ct+d}$. 

We have: $f(z_3)=f(\infty)=\infty$, which implies $c=0,\ d=1$. Further, substituting the known values we get
$$f(z_4)=0=\frac{az_4+b}{cz_4+d}=az_4+b,$$
$$f(z_5)=1= \frac{az_5+b}{cz_5+d}=az_5+b.$$
Then $a=\frac{1}{z_5-z_4},\ b=\frac{z_4}{z_4-z_5}$. We obtain, $$f(t)=\frac{z_4-t}{z_4-z_5}.$$
Therefore, 
$$y_1=f(0)=\frac{z_4}{z_4-z_5} \mbox{ and } y_2=f(1)=\frac{z_4-1}{z_4-z_5}.$$

2. Let us consider two cells of maximal dimension, such that the cell $K_{12|345}$ is their common boundary. Orientations of these maximal dimension cells provided by the standard orientation of the plane  $(y_1,y_2)$ are opposite and thus induce opposite orientations on  $K_{12|345}$. We have to understand if the orientations of these two cells provided by the orientation of the plane $(z_1,z_2)$ are the same or opposite. Thus we have to check if the orientation of the maximal dimension cells provided by the orientation of  $(y_1,y_2)$ coincides with the orientation provided by the orientation of  $(z_4,z_5)$.  

We compute the Jacobian of coordinate change from $(z_4,z_5)$ to $(y_1,y_2)$:
$$J=\det \left( \begin{array}{cc} \frac{\partial y_1}{\partial z_4} &  \frac{\partial y_1}{\partial z_5} \\  \frac{\partial y_2}{\partial z_4}  &  \frac{\partial y_2}{\partial z_5} \end{array}\right).$$
Direct computations which we omit to shorter the text show that
$$J= \det \left( \begin{array}{cc} \frac{-z_5}{(z_4-z_5)^2} &  \frac{z_4}{(z_4-z_5)^2} \\  \frac{1-z_5}{(z_4-z_5)^2} &  \frac{1+z_4}{(z_4-z_5)^2}\end{array}\right) =  \frac{1}{(z_4-z_5)^4}(z_5-z_4).$$

3. The cell $K_{12|345}$ is the common boundary of the following two cells:

\begin{center}
\setlength{\unitlength}{0.5pt}
\begin{picture}(360,140)
\put(0,0){\line(1,0){80}}
\put(280,0){\line(1,0){80}}
\qbezier{(-30,70)(-15,35)(0,0)}
\qbezier{(110,70)(95,35)(80,0)}
\qbezier{(-30,70)(5,90)(40,110)}
\qbezier{(40,110)(75,90)(110,70)}
\qbezier{(250,70)(265,35)(280,0)}
\qbezier{(390,70)(375,35)(360,0)}
\qbezier{(250,70)(285,90)(320,110)}
\qbezier{(320,110)(355,90)(390,70)}
\put(40,-24){$4$}
\put(103,27){$3$}
\put(75,90){$2$}
\put(-5,90){$1$}
\put(-32,27){$5$}
\put(320,-24){$4$}
\put(383,27){$3$}
\put(355,90){$1$}
\put(275,90){$2$}
\put(248,27){$5$}
\end{picture}
\end{center}

These cells correspond to the following orders of the marked points:

\begin{center}
\begin{picture}(360,10)
\put(-5,10){\line(1,0){85}}
\put(275,10){\line(1,0){85}}
\put(0,6){$\bullet$}
\put(20,6){$\bullet$}
\put(40,6){$\bullet$}
\put(60,6){$\bullet$}
\put(80,6){$\bullet$}
\put(0,-4){$z_4$}
\put(20,-4){$z_5$}
\put(40,-4){$0$}
\put(60,-4){$1$}
\put(80,-4){$\infty$}
\put(280,6){$\bullet$}
\put(300,6){$\bullet$}
\put(320,6){$\bullet$}
\put(340,6){$\bullet$}
\put(360,6){$\bullet$}
\put(280,-4){$0$}
\put(300,-4){$1$}
\put(320,-4){$z_5$}
\put(340,-4){$z_4$}
\put(360,-4){$\infty$}
\end{picture}
\end{center}

correspondingly. 

4. So, we have that  $z_4<z_5<0$ for the cell marked by the left pentagon, and $1<z_5<z_4<\infty$ for the cell marked by the right pentagon. Then for the left cell we get $J>0$, and for the right one $J<0$. Since in the parametrization $(y_1,y_2)$ these cells have opposite orientations, and the Jacobians have the opposite signs, we get that in the parametrization $(z_4,z_5)$ these cells have the same orientation. Hence, the cell $K_{12|345}$ is included twice to the expression for  $\Thh(\overline{{\mathcal M}_{0,5}^\R})$. Multiplying the sum of boundary cells by $\frac12$ we get that the cell $K_{12|345}$ goes with the coefficient 1, which remains fixed modulus 2. Therefore, the class $\Thh(\overline{{\mathcal M}_{0,5}^\R})$ contains the cell labeled by the pentagon $K_{12|345}$.
\end{proof}

\begin{lemma} \label{K12435} Boundary cells labeled by the pentagons

\begin{center}
\setlength{\unitlength}{0.5pt}
\begin{picture}(360,140)
\put(0,0){\line(1,0){80}}
\put(280,0){\line(1,0){80}}
\qbezier{(-30,70)(-15,35)(0,0)}
\qbezier{(110,70)(95,35)(80,0)}
\qbezier{(-30,70)(5,90)(40,110)}
\qbezier{(40,110)(75,90)(110,70)}
\qbezier{(250,70)(265,35)(280,0)}
\qbezier{(390,70)(375,35)(360,0)}
\qbezier{(250,70)(285,90)(320,110)}
\qbezier{(320,110)(355,90)(390,70)}
\qbezier{(-30,70)(25,70)(110,70)}
\qbezier{(250,70)(305,70)(390,70)}
\put(40,-24){$3$}
\put(103,27){$4$}
\put(75,90){$2$}
\put(-5,90){$1$}
\put(-32,27){$5$}
\put(320,-24){$5$}
\put(383,27){$3$}
\put(355,90){$2$}
\put(275,90){$1$}
\put(248,27){$4$}
\put(-15,130){$K_{12|435}$}
\put(265,130){$K_{12|354}$}
\end{picture}
\end{center}
are in the class $\Thh(\overline{{\mathcal M}_{0,5}^\R})$, which is Poincar\'e dual to the first  Stiefel-Whitney class of $\overline{{\mathcal M}_{0,5}^\R}$.
\end{lemma}

\begin{proof}
1. Our proof is similar to the proof of Lemma \ref{K12345} and we use the coordinates $(y_1,y_2)$ on   ${{\mathcal M}_{0,5}^\R}$ introduced in that proof. The cell labeled by $K_{12|435}$ is the common bound of the cells labeled by the pentagons

\begin{center}
\setlength{\unitlength}{0.5pt}
\begin{picture}(360,140)
\put(0,0){\line(1,0){80}}
\put(280,0){\line(1,0){80}}
\qbezier{(-30,70)(-15,35)(0,0)}
\qbezier{(110,70)(95,35)(80,0)}
\qbezier{(-30,70)(5,90)(40,110)}
\qbezier{(40,110)(75,90)(110,70)}
\qbezier{(250,70)(265,35)(280,0)}
\qbezier{(390,70)(375,35)(360,0)}
\qbezier{(250,70)(285,90)(320,110)}
\qbezier{(320,110)(355,90)(390,70)}
\put(40,-24){$3$}
\put(103,27){$4$}
\put(75,90){$2$}
\put(-5,90){$1$}
\put(-32,27){$5$}
\put(320,-24){$3$}
\put(383,27){$4$}
\put(355,90){$1$}
\put(275,90){$2$}
\put(248,27){$5$}
\end{picture}
\end{center}

The value of the Jacobian for the change of coordinates from $(z_4,z_5)$ to $(y_1,y_2)$ by Lemma  \ref{K12345} is equal to   $J=\frac{z_5-z_4}{(z_4-z_5)^4}$. For all elements of these two cells we have the following order of points on curves:

\begin{center}
\begin{picture}(360,10)
\put(-5,0){\line(1,0){85}}
\put(275,0){\line(1,0){85}}
\put(0,-4){$\bullet$}
\put(20,-4){$\bullet$}
\put(40,-4){$\bullet$}
\put(60,-4){$\bullet$}
\put(80,-4){$\bullet$}
\put(0,-14){$z_5$}
\put(20,-14){$z_4$}
\put(40,-14){$0$}
\put(60,-14){$1$}
\put(80,-14){$\infty$}
\put(280,-4){$\bullet$}
\put(300,-4){$\bullet$}
\put(320,-4){$\bullet$}
\put(340,-4){$\bullet$}
\put(360,-4){$\bullet$}
\put(280,-14){$0$}
\put(300,-14){$1$}
\put(320,-14){$z_4$}
\put(340,-14){$z_5$}
\put(360,-14){$\infty$}
\end{picture}
\end{center}
  
So, for the cell marked by the left pentagon we get $J<0$, and for the right one: $J>0$. Since in the coordinates $(y_1,y_2)$ these cells have opposite orientations, in the coordinates  $(z_4,z_5)$ they have the same orientation. Hence their common bound, i.e., the cell labeled by $K_{12|435}$, is in the class $\Thh(\overline{{\mathcal M}_{0,5}^\R})$. 

2. Similarly, for the cell labeled by  $K_{12|354}$ the Jacobian value is the same, however the points on the curves in the corresponding cells of the maximal dimension are ordered as follows:

\begin{center}
\begin{picture}(360,10)
\put(-5,0){\line(1,0){85}}
\put(275,0){\line(1,0){85}}
\put(0,-4){$\bullet$}
\put(20,-4){$\bullet$}
\put(40,-4){$\bullet$}
\put(60,-4){$\bullet$}
\put(80,-4){$\bullet$}
\put(0,-14){$z_5$}
\put(20,-14){$z_4$}
\put(40,-14){$0$}
\put(60,-14){$1$}
\put(80,-14){$\infty$}
\put(280,-4){$\bullet$}
\put(300,-4){$\bullet$}
\put(320,-4){$\bullet$}
\put(340,-4){$\bullet$}
\put(360,-4){$\bullet$}
\put(280,-14){$0$}
\put(300,-14){$1$}
\put(320,-14){$z_4$}
\put(340,-14){$z_5$}
\put(360,-14){$\infty$}
\end{picture}
\end{center}
  
Thus, for one of them we get $J<0$, and for another one $J>0$.  Since in the coordinates $(y_1,y_2)$ these cells have opposite orientations, we get that in the coordinates  $(z_4,z_5)$ they have the same orientation. Hence their common bound, i.e., the cell labeled by $K_{12|354}$, is in the class $\Thh(\overline{{\mathcal M}_{0,5}^\R})$. 
\end{proof}

\begin{lemma} \label{K13} The boundary cells labeled by the pentagons 

\begin{center}
\setlength{\unitlength}{0.5pt}
\begin{picture}(640,140)
\put(0,0){\line(1,0){80}}
\put(280,0){\line(1,0){80}}
\qbezier{(-30,70)(-15,35)(0,0)}
\qbezier{(110,70)(95,35)(80,0)}
\qbezier{(-30,70)(5,90)(40,110)}
\qbezier{(40,110)(75,90)(110,70)}
\qbezier{(250,70)(265,35)(280,0)}
\qbezier{(390,70)(375,35)(360,0)}
\qbezier{(250,70)(285,90)(320,110)}
\qbezier{(320,110)(355,90)(390,70)}
\qbezier{(-30,70)(25,70)(110,70)}
\qbezier{(250,70)(305,70)(390,70)}
\put(40,-24){$4$}
\put(103,27){$5$}
\put(75,90){$3$}
\put(-5,90){$1$}
\put(-32,27){$2$}
\put(320,-24){$5$}
\put(383,27){$4$}
\put(355,90){$3$}
\put(275,90){$1$}
\put(248,27){$2$}
\put(560,0){\line(1,0){80}}
\qbezier{(530,70)(545,35)(560,0)}
\qbezier{(670,70)(655,35)(640,0)}
\qbezier{(530,70)(565,90)(600,110)}
\qbezier{(600,110)(635,90)(670,70)}
\qbezier{(-30,70)(25,70)(110,70)}
\qbezier{(530,70)(585,70)(670,70)}
\put(600,-24){$2$}
\put(663,27){$5$}
\put(635,90){$3$}
\put(555,90){$1$}
\put(528,27){$4$}
\put(-15,130){$K_{13|542}$}
\put(265,130){$K_{13|452}$}
\put(545,130){$K_{13|524}$}
\end{picture}
\end{center}
are in the class $\Thh(\overline{{\mathcal M}_{0,5}^\R})$.
\end{lemma}

\begin{proof}
1. For the investigation of these  3 cells we consider new coordinates on ${\mathcal M}_{0,5}^{\mathbb R}$. As in Definition \ref{DefKoord} we construct the coordinatization map 
 $\varphi_2: {\mathcal M}_{0,5}^{{\mathbb R}} \to {\mathbb R}^2$. Let $({\mathbb P}_1(\R),x_1,\ldots, x_5)\in {{\mathcal M}_{0,5}^{\mathbb R}}$ be a point of the moduli space ${\mathcal M}_{0,5}^{{\mathbb R}}$. We choose the parametrization $x$ of ${\mathbb P}_1(\R) $ such that $x_2 =0, x_4=1, x_5=\infty$ and set  $\varphi_2({\mathbb P}_1(\R),x_1,\ldots,x_5)=(x_1,x_3)$. 

Let us find the expression for the coordinates  $(x_1,x_3)$ via the coordinates  $(z_4,z_5)$. To do this, we set to each of the marked points  $i$ its $x$-coordinates and $z$-coordinates on ${\mathbb P}_1(\R) $. Then we will find the rational-fractional function which maps $z$-coordinates to  $x$-coordinates. The condition is
$$
\begin{array}{cccccc} i& 1&2&3&4&5 \\ x-\mbox{coordinate } & x_1& 0 & x_3 & 1 &\infty \\ z-\mbox{coordinate} &0&1&\infty& z_4&z_5 \end{array}.$$ 

We find rational-fractional transformation $f(t)$ by the conditions $f(1)=0$, $f(z_4)=1$, $f(z_5)=\infty$.  Then
$$f(t)=\frac{z_4-z_5}{z_4-1}\cdot \frac{t-1}{t-z_5}.$$
Hence in this case the Jacobian of the coordinate change is equal to 
$$J=\det\left(\begin{array}{cc} \frac{z_5-1}{z_5(z_4-1)^2} & \frac{-z_4}{(z_4-1)z_5^2} \\ \frac{z_5-1}{(z_4-1)^2} & \frac{1}{1-z_4} \end{array} \right) = \frac{(1-z_5)(z_5-z_4)}{(z_4-1)^3z_5^2}.$$

2. The cell labeled by $K_{13|542}$ is the common boundary of the cells labeled by:

\begin{center}
\setlength{\unitlength}{0.5pt}
\begin{picture}(360,140)
\put(0,0){\line(1,0){80}}
\put(280,0){\line(1,0){80}}
\qbezier{(-30,70)(-15,35)(0,0)}
\qbezier{(110,70)(95,35)(80,0)}
\qbezier{(-30,70)(5,90)(40,110)}
\qbezier{(40,110)(75,90)(110,70)}
\qbezier{(250,70)(265,35)(280,0)}
\qbezier{(390,70)(375,35)(360,0)}
\qbezier{(250,70)(285,90)(320,110)}
\qbezier{(320,110)(355,90)(390,70)}
\put(40,-24){$4$}
\put(103,27){$5$}
\put(75,90){$3$}
\put(-5,90){$1$}
\put(-32,27){$2$}
\put(320,-24){$4$}
\put(383,27){$5$}
\put(355,90){$1$}
\put(275,90){$3$}
\put(248,27){$2$}
\end{picture}
\end{center}

The following variants of marked point order correspond to these cells:

\begin{center}
\begin{picture}(360,10)
\put(-5,0){\line(1,0){85}}
\put(275,0){\line(1,0){85}}
\put(0,-4){$\bullet$}
\put(20,-4){$\bullet$}
\put(40,-4){$\bullet$}
\put(60,-4){$\bullet$}
\put(80,-4){$\bullet$}
\put(0,-14){$0$}
\put(20,-14){$1$}
\put(40,-14){$z_4$}
\put(60,-14){$z_5$}
\put(80,-14){$\infty$}
\put(280,-4){$\bullet$}
\put(300,-4){$\bullet$}
\put(320,-4){$\bullet$}
\put(340,-4){$\bullet$}
\put(360,-4){$\bullet$}
\put(280,-14){$0$}
\put(300,-14){$z_5$}
\put(320,-14){$z_4$}
\put(340,-14){$1$}
\put(360,-14){$\infty$}
\end{picture}
\end{center}

\bigskip

Then $0<1<z_4<z_5<\infty$, hence, $J<0$ for the cell labeled by the left pentagon. Also $0<z_5<z_4<1<\infty$  and $J>0$ for the cell labeled by the right pentagon. As in the previous lemmas we get that the cell labeled by $K_{13|542}$ is in the class $\Thh(\overline{{\mathcal M}_{0,5}^{\R}})$.

4. Similarly, the cell labeled by $K_{13|452}$ is the common boundary of the cells labeled by

\begin{center}
\setlength{\unitlength}{0.5pt}
\begin{picture}(360,140)
\put(0,0){\line(1,0){80}}
\put(280,0){\line(1,0){80}}
\qbezier{(-30,70)(-15,35)(0,0)}
\qbezier{(110,70)(95,35)(80,0)}
\qbezier{(-30,70)(5,90)(40,110)}
\qbezier{(40,110)(75,90)(110,70)}
\qbezier{(250,70)(265,35)(280,0)}
\qbezier{(390,70)(375,35)(360,0)}
\qbezier{(250,70)(285,90)(320,110)}
\qbezier{(320,110)(355,90)(390,70)}
\put(40,-24){$5$}
\put(103,27){$4$}
\put(75,90){$3$}
\put(-5,90){$1$}
\put(-32,27){$2$}
\put(320,-24){$5$}
\put(383,27){$4$}
\put(355,90){$1$}
\put(275,90){$3$}
\put(248,27){$2$}
\end{picture}
\end{center}

i.e., the cells containing curves with the following order of the marked points

\begin{center}
\begin{picture}(360,10)
\put(-5,0){\line(1,0){85}}
\put(275,0){\line(1,0){85}}
\put(0,-4){$\bullet$}
\put(20,-4){$\bullet$}
\put(40,-4){$\bullet$}
\put(60,-4){$\bullet$}
\put(80,-4){$\bullet$}
\put(0,-14){$0$}
\put(20,-14){$1$}
\put(40,-14){$z_5$}
\put(60,-14){$z_4$}
\put(80,-14){$\infty$}
\put(280,-4){$\bullet$}
\put(300,-4){$\bullet$}
\put(320,-4){$\bullet$}
\put(340,-4){$\bullet$}
\put(360,-4){$\bullet$}
\put(280,-14){$0$}
\put(300,-14){$z_4$}
\put(320,-14){$z_5$}
\put(340,-14){$1$}
\put(360,-14){$\infty$}
\end{picture}
\end{center}
correspondingly.

Then  $J>0$ for the cell labeled by the left pentagon and  $J<0$ for the cell labeled by the right pentagon. Hence, the cell marked by  $K_{13|452}$ is in the class $\Thh(\overline{{\mathcal M}_{0,5}^{\R}})$.

5. Finally, the cell labeled by  $K_{13|524}$ is the common boundary of two cells, each of them consists of the curves with the following order of marked points:

\begin{center}
\begin{picture}(360,10)
\put(-5,0){\line(1,0){85}}
\put(275,0){\line(1,0){85}}
\put(0,-4){$\bullet$}
\put(20,-4){$\bullet$}
\put(40,-4){$\bullet$}
\put(60,-4){$\bullet$}
\put(80,-4){$\bullet$}
\put(0,-14){$0$}
\put(20,-14){$z_4$}
\put(40,-14){$1$}
\put(60,-14){$z_5$}
\put(80,-14){$\infty$}
\put(280,-4){$\bullet$}
\put(300,-4){$\bullet$}
\put(320,-4){$\bullet$}
\put(340,-4){$\bullet$}
\put(360,-4){$\bullet$}
\put(280,-14){$0$}
\put(300,-14){$z_5$}
\put(320,-14){$1$}
\put(340,-14){$z_4$}
\put(360,-14){$\infty$}
\end{picture}
\end{center}
correspondingly.

Then  $J>0$ for the cell represented on the left-hand side and  $J<0$ for the cell represented on the right-hand side. Therefore, the cell labeled by  $K_{13|524}$ is in the class $\Thh(\overline{{\mathcal M}_{0,5}^{\R}})$.
\end{proof}

\begin{lemma} \label{K23} All three boundary cells labeled by the pentagons with the diagonal cutting the sides marked by 2 and 3, i.e., labeled by the pentagons of the form

\begin{center}
\setlength{\unitlength}{0.5pt}
\begin{picture}(640,140)
\put(0,0){\line(1,0){80}}
\put(280,0){\line(1,0){80}}
\qbezier{(-30,70)(-15,35)(0,0)}
\qbezier{(110,70)(95,35)(80,0)}
\qbezier{(-30,70)(5,90)(40,110)}
\qbezier{(40,110)(75,90)(110,70)}
\qbezier{(250,70)(265,35)(280,0)}
\qbezier{(390,70)(375,35)(360,0)}
\qbezier{(250,70)(285,90)(320,110)}
\qbezier{(320,110)(355,90)(390,70)}
\qbezier{(-30,70)(25,70)(110,70)}
\qbezier{(250,70)(305,70)(390,70)}
\put(40,-24){$5$}
\put(103,27){$4$}
\put(75,90){$3$}
\put(-5,90){$2$}
\put(-32,27){$1$}
\put(320,-24){$4$}
\put(383,27){$5$}
\put(355,90){$3$}
\put(275,90){$2$}
\put(248,27){$1$}
\put(560,0){\line(1,0){80}}
\qbezier{(530,70)(545,35)(560,0)}
\qbezier{(670,70)(655,35)(640,0)}
\qbezier{(530,70)(565,90)(600,110)}
\qbezier{(600,110)(635,90)(670,70)}
\qbezier{(-30,70)(25,70)(110,70)}
\qbezier{(530,70)(585,70)(670,70)}
\put(600,-24){$1$}
\put(663,27){$5$}
\put(635,90){$3$}
\put(555,90){$2$}
\put(528,27){$4$}
\put(-15,130){$K_{23|451}$}
\put(265,130){$K_{23|541}$}
\put(545,130){$K_{23|514}$}
\end{picture}
\end{center}
are in the class $\Thh(\overline{{\mathcal M}_{0,5}^{\R}})$.
\end{lemma}

\begin{proof}
1. Points of these three cells are the stable curves such that the marked points with labels 2 and 3 are on the separate component. Thus to investigate these cells we have to consider the coordinates on  ${\mathcal M}_{0,5}^{\mathbb R}$ in which gluing of the points $z_2$ and $z_3$ can be seen. For this similar to Definition \ref{DefKoord} we construct the coordinatization map  $\varphi_3: {\mathcal M}_{0,5}^{\mathbb R} \to {\mathbb R}^2$. Let $({\mathbb P}_1(\R),u_1,\ldots, u_5)\in {{\mathcal M}_{0,5}^{\mathbb R}}$  be an element of the moduli space ${\mathcal M}_{0,5}^{\mathbb R}$. We choose the parametrization $u$ on the curve ${\mathbb P}_1(\R) $ such that  $u_1 =0, u_4=1, u_5=\infty$. Let us set $\varphi_3({\mathbb P}_1(\R),u_1,\ldots,u_5)=(u_2,u_3)$. 

To find the transition functions from $(z_4,z_5)$ to  $(u_2,u_3)$, we write the correspondence between different coordinates of the points:
$$
\begin{array}{cccccc} \mbox{ point } i& 1&2&3&4&5 \\u-\mbox{coordinate } & 0& u_2& u_3 & 1 &\infty \\ z-\mbox{coordinate } &0&1&\infty& z_4&z_5 \end{array}.$$ 
Then the corresponding linear-fractional transformation has the form
$$f(t)=\frac{z_4-z_5}{z_4}\frac{t}{t-z_5}$$
and one can easily write the explicit formulas for $u_2$ and $u_3$ as the functions of $z_4,z_5$.

Therefore, the Jacobian of this change of the coordinates has the form
$$J=\det\left(\begin{array}{cc} \frac{z_5}{(1-z_5)z_4^2} & \frac{z_4-1}{z_4(1-z_5)^2} \\ \frac{z_5}{z_4^2} & \frac{-1}{z_4} \end{array} \right) = \frac{z_5(z_5-z_4)}{z_4^3(1-z_5)^2}.$$

2. The cell labeled by $K_{23|451}$ is the common boundary of the two cells labeled by

\begin{center}
\setlength{\unitlength}{0.5pt}
\begin{picture}(360,140)
\put(0,0){\line(1,0){80}}
\put(280,0){\line(1,0){80}}
\qbezier{(-30,70)(-15,35)(0,0)}
\qbezier{(110,70)(95,35)(80,0)}
\qbezier{(-30,70)(5,90)(40,110)}
\qbezier{(40,110)(75,90)(110,70)}
\qbezier{(250,70)(265,35)(280,0)}
\qbezier{(390,70)(375,35)(360,0)}
\qbezier{(250,70)(285,90)(320,110)}
\qbezier{(320,110)(355,90)(390,70)}
\put(40,-24){$5$}
\put(103,27){$4$}
\put(75,90){$3$}
\put(-5,90){$2$}
\put(-32,27){$1$}
\put(320,-24){$5$}
\put(383,27){$4$}
\put(355,90){$2$}
\put(275,90){$3$}
\put(248,27){$1$}
\end{picture}
\end{center}

These cells correspond to the following orders of marked points on the curves:

\begin{center}
\begin{picture}(360,10)
\put(-5,0){\line(1,0){85}}
\put(275,0){\line(1,0){85}}
\put(0,-4){$\bullet$}
\put(20,-4){$\bullet$}
\put(40,-4){$\bullet$}
\put(60,-4){$\bullet$}
\put(80,-4){$\bullet$}
\put(0,-14){$z_4$}
\put(20,-14){$z_5$}
\put(40,-14){$0$}
\put(60,-14){$1$}
\put(80,-14){$\infty$}
\put(280,-4){$\bullet$}
\put(300,-4){$\bullet$}
\put(320,-4){$\bullet$}
\put(340,-4){$\bullet$}
\put(360,-4){$\bullet$}
\put(280,-14){$0$}
\put(300,-14){$z_5$}
\put(320,-14){$z_4$}
\put(340,-14){$1$}
\put(360,-14){$\infty$}
\end{picture}
\end{center}

\bigskip

Then $z_4<z_5<0<1<\infty$, hence, $J<0$ for the cell labeled by the left pentagon, and $0<z_5<z_4<1<\infty$,  $J>0$ for the cell labeled by the right pentagon. As in the previous cases we get that the cell labeled by $K_{23|451}$ is in the class $\Thh(\overline{{\mathcal M}_{0,5}^\R})$.

4. Similarly, the cell labeled by $K_{23|541}$ is the common bound of the cells labeled by

\begin{center}
\setlength{\unitlength}{0.5pt}
\begin{picture}(360,140)
\put(0,0){\line(1,0){80}}
\put(280,0){\line(1,0){80}}
\qbezier{(-30,70)(-15,35)(0,0)}
\qbezier{(110,70)(95,35)(80,0)}
\qbezier{(-30,70)(5,90)(40,110)}
\qbezier{(40,110)(75,90)(110,70)}
\qbezier{(250,70)(265,35)(280,0)}
\qbezier{(390,70)(375,35)(360,0)}
\qbezier{(250,70)(285,90)(320,110)}
\qbezier{(320,110)(355,90)(390,70)}
\put(40,-24){$4$}
\put(103,27){$5$}
\put(75,90){$3$}
\put(-5,90){$2$}
\put(-32,27){$1$}
\put(320,-24){$4$}
\put(383,27){$5$}
\put(355,90){$2$}
\put(275,90){$3$}
\put(248,27){$1$}
\end{picture}
\end{center}

These cells correspond to the following orders of marked points on the curves:

\begin{center}
\begin{picture}(360,15)
\put(-5,10){\line(1,0){85}}
\put(275,10){\line(1,0){85}}
\put(0,6){$\bullet$}
\put(20,6){$\bullet$}
\put(40,6){$\bullet$}
\put(60,6){$\bullet$}
\put(80,6){$\bullet$}
\put(0,-4){$z_5$}
\put(20,-4){$z_4$}
\put(40,-4){$0$}
\put(60,-4){$1$}
\put(80,-4){$\infty$}
\put(280,6){$\bullet$}
\put(300,6){$\bullet$}
\put(320,6){$\bullet$}
\put(340,6){$\bullet$}
\put(360,6){$\bullet$}
\put(280,-4){$0$}
\put(300,-4){$z_4$}
\put(320,-4){$z_5$}
\put(340,-4){$1$}
\put(360,-4){$\infty$}
\put(180,10){\large{и}}
\end{picture}
\end{center}

Then  $J<0$ for the cell labeled by the left pentagon, and  $J>0$ for the cell labeled by the right pentagon. Thus, the cell labeled by $K_{23|541}$ is in the class $\Thh(\overline{{\mathcal M}_{0,5}^\R})$.

5. Finally, the cell labeled by $K_{23|514}$  is the common bound of the cells which consist of the curves with the following order of marked points:

\begin{center}
\begin{picture}(360,15)
\put(-5,10){\line(1,0){85}}
\put(275,10){\line(1,0){85}}
\put(0,6){$\bullet$}
\put(20,6){$\bullet$}
\put(40,6){$\bullet$}
\put(60,6){$\bullet$}
\put(80,6){$\bullet$}
\put(0,-4){$z_5$}
\put(20,-4){$0$}
\put(40,-4){$z_4$}
\put(60,-4){$1$}
\put(80,-4){$\infty$}
\put(280,6){$\bullet$}
\put(300,6){$\bullet$}
\put(320,6){$\bullet$}
\put(340,6){$\bullet$}
\put(360,6){$\bullet$}
\put(280,-4){$z_4$}
\put(300,-4){$0$}
\put(320,-4){$z_5$}
\put(340,-4){$1$}
\put(360,-4){$\infty$}
\end{picture}
\end{center}

Then  $J>0$ for the cell represented on the left figure, and  $J<0$ for the cell represented on the right figure. Hence, the cell labeled by $K_{23|514}$ is in the class $\Thh(\overline{{\mathcal M}_{0,5}^\R})$.
\end{proof}

\begin{corollary}
Poincar\'e dual to the first Stiefel-Whitney class of $\overline{{\mathcal M}_{0,5}^\R}$, the class  $\Thh(\overline{{\mathcal M}_{0,5}^\R})$ contains all cells labeled by the pentagons of the form

\begin{center}
\setlength{\unitlength}{0.5pt}
\begin{picture}(80,100)
\put(0,0){\line(1,0){80}}
\qbezier{(-30,70)(-15,35)(0,0)}
\qbezier{(110,70)(95,35)(80,0)}
\qbezier{(-30,70)(5,90)(40,110)}
\qbezier{(40,110)(75,90)(110,70)}
\qbezier{(-30,70)(25,70)(110,70)}
\put(75,90){$i$}
\put(-5,90){$j$}
\end{picture}
\end{center}

where $1\le i,j\le 3$, and only these cells.
\end{corollary}

\subsection{ Computation of  $\Th(\overline{{\mathcal M}_{0,n}^\R})$ for  $n\ge 6$}

In this section we introduce special coordinates, which provide the possibility to investigate the situation nearby some  interesting for us boundaries. In order to do this we draw the curve ${\mathbb P}_1(\R)$  in the form of hyperbola  $xy=\ep$. Approaching the boundary correspond to taking the limit $\ep\to 0$ under the fixed coordinates  $x$ or $y$ of marked points on the curve. Going to the boundary transforms the hyperbola to a pair of intersecting lines. As  local coordinates in a neighborhood  of the boundary we use the parameter  $\varepsilon$, $x$-coordinates of the points with positive $x$-coordinates, and $y$-coordinates for the points with negative $y$-coordinates. Coordinates of some 4 points we fix, then the rest  $n-4$ points and  $\ep$ provide exactly the required $n-3$ coordinates. Below, see the formulas (\ref{eq:J>}) and (\ref{eq:J}), we show that for sufficiently small $\ep\ne 0$ the Jacobian of the corresponding transition function can not be zero, thus this parametrization indeed provides local coordinates on ${\mathcal M}_{0,n}^\R$.

\begin{definition}
Let us define the coordinates in the following way. Let the boundary under consideration is labeled by the polygon $K_{l_1,\ldots,l_m\vert k_1,\ldots,k_{n-m}}$, which is drawn at the middle part of Figure \ref{param1}. This boundary is a common boundary of the cells labeled by the polygons $K_{l_1,\ldots,l_m, k_1,\ldots,k_{n-m}}$ and  $K_{l_1,\ldots,l_m, k_{n-m},\ldots,k_1}$, drawn on the left side and the right side at Figure \ref{param1}, correspondingly. The cell labeled by  $K_{l_1,\ldots,l_m, k_1,\ldots,k_{n-m}}$ is parametrized by  $\ep> 0$ and corresponds to the hyperbola lying in the intersection of right and upper semi-planes and in the intersection of left and lower semi-planes at Figure \ref{param}. Similarly, the cell labeled by  $K_{l_1,\ldots,l_m, k_{n-m},\ldots,k_1}$  is parametrized by $\ep< 0$ and corresponds to the hyperbola lying in the intersection of right and lower semi-planes and in the intersection of left and upper semi-planes at Figure \ref{param}. We choose 2 points in each of two parts of the polygon  $K_{l_1,\ldots,l_m\vert k_1,\ldots,k_{n-m}}$, separated by the diagonal. We denote the chosen points by  $l_{i_1},l_{i_2}$ and $k_{j_1},k_{j_2}$, correspondingly. Let us fix their coordinates: abscissas $x(l_{i_1})=1, x(l_{i_2})=2$ and ordinates $ y(k_{j_1})=-1,y(k_{j_2})=-2$. Then the coordinates in the neighborhood of this boundary are the list of abscissas of the points $\{l_1,\ldots,l_m\}\setminus \{l_{i_1},l_{i_2}\}$, the list of ordinates of the points $\{ k_1,\ldots,k_{n-m}\}\setminus\{k_{j_1},k_{j_2}\}$ and the value of the parameter  $\ep$.
\end{definition}

\begin{figure}[h]
\centering\epsfig{figure=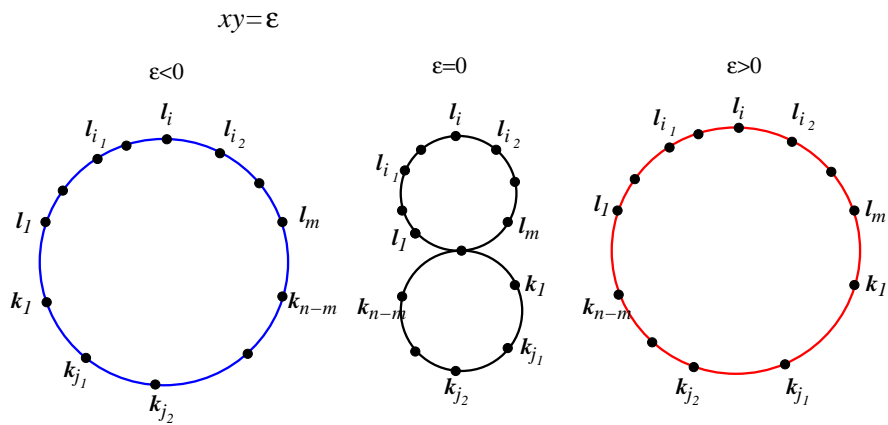,width=.7\linewidth}
\caption{$xy=\varepsilon$} \label{param1}
\end{figure}

\begin{re}
Note that the points   $l_{i_1},l_{i_2}$ и $k_{j_1},k_{j_2}$  are not necessary neighbour, the first, the last, etc., there are no restrictions on their positioning on the polygon, except that the first and the second pairs are divided by the diagonal.
\end{re}

\begin{lemma} \label{lem123}
Let $\sigma\in S_3$ be a permutation of the points $0,1$, and $\infty$. Let  
$\varphi_\sigma: {{\mathcal M}_{0,n}^{\mathbb R}} \to {\mathbb R}^{n-3}$ be a coordinatization map, constructed similarly to Definition  \ref{DefKoord} in the following way. If  $({\mathbb P}_1(\R),w_1,\ldots, w_n)\in {{\mathcal M}_{0,n}^{\mathbb R}}$ is a point of the moduli space  ${\mathcal M}_{0,n}^\R$, and parametrization on the curve  ${\mathbb P}_1(\R) $ is chosen in such a way that $w_3 =\sigma(\infty), w_2=\sigma(1),w_1=\sigma(0)$, then $\varphi_\sigma({\mathbb P}_1(\R),w_1,\ldots,w_n)=(w_4,\ldots,w_n)$. Let the orientation of all cells is provided by the standard orientation of the space  $\R^{n-3}=\{(w_4,\ldots,w_n)\}$.

Then either the orientation of all cells induced by the coordinates  $\varphi_\sigma$ coincides with the orientation of these cells induced by the coordinates  $\varphi$ or for all cells these orientations are opposite. \end{lemma}
\begin{proof}
All permutations of the values 0, 1, $\infty$ on the projective line ${\mathbb P}_1(\R)$  are compositions of the transformations  $f_1(t)=1-t$ and $f_2(t)=\frac1t$. Hence change of the coordinates  $z$ to the coordinates $w$ are compositions of the maps $z_i\to 1- w_i$ for all $i=4,\ldots,n$ and $z_i\to \frac1{ w_i}$ for all $i=4,\ldots,n$. Since these transformations are diagonal, the corresponding Jacobians are  
$$J_1=\prod_{k=4}^n \frac{\partial(1-z_k)}{\partial z_k}=(-1)^{n-3}$$
and
$$J_2=\prod_{k=4}^n \frac{\partial(\frac1{z_k})}{\partial z_k}=\prod_{k=4}^n \frac{-1}{z_k^2}=(-1)^{n-3}\prod_{k=4}^n \frac1{z_k^2}.$$
So, in both cases orientation does not depend on the cell, but depends on the oddity of $n$. If $n$ is even, then change of the coordinates converts the orientation of any cell to the opposite one. If $n$ is odd then the orientation of any cell remains the same. 

Composition of transformations corresponds to the multiplication of the Jacobians, hence in all the cases either the change of the coordinates changes orientation of all cells or it leaves orientations of all cells fixed. 
\end{proof}

\subsubsection{The case when the points  1, 2, 3 are in the same component of the boundary}


In this case corresponding cells and coordinates on them are shown at Figure  \ref{param0}, \ref{param} and \ref{123}.
\begin{figure}[h]
\centering\epsfig{figure=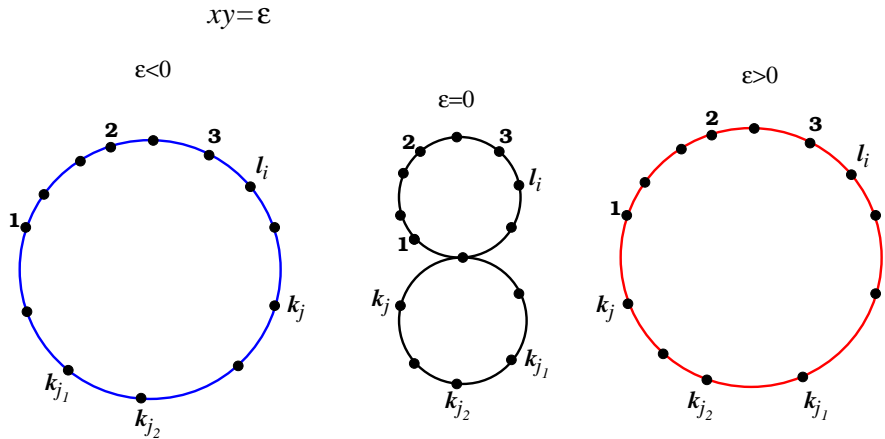,width=.7\linewidth}
\caption{$xy=\varepsilon$} \label{param0}
\end{figure}
\begin{figure}[h]
\centering\epsfig{figure=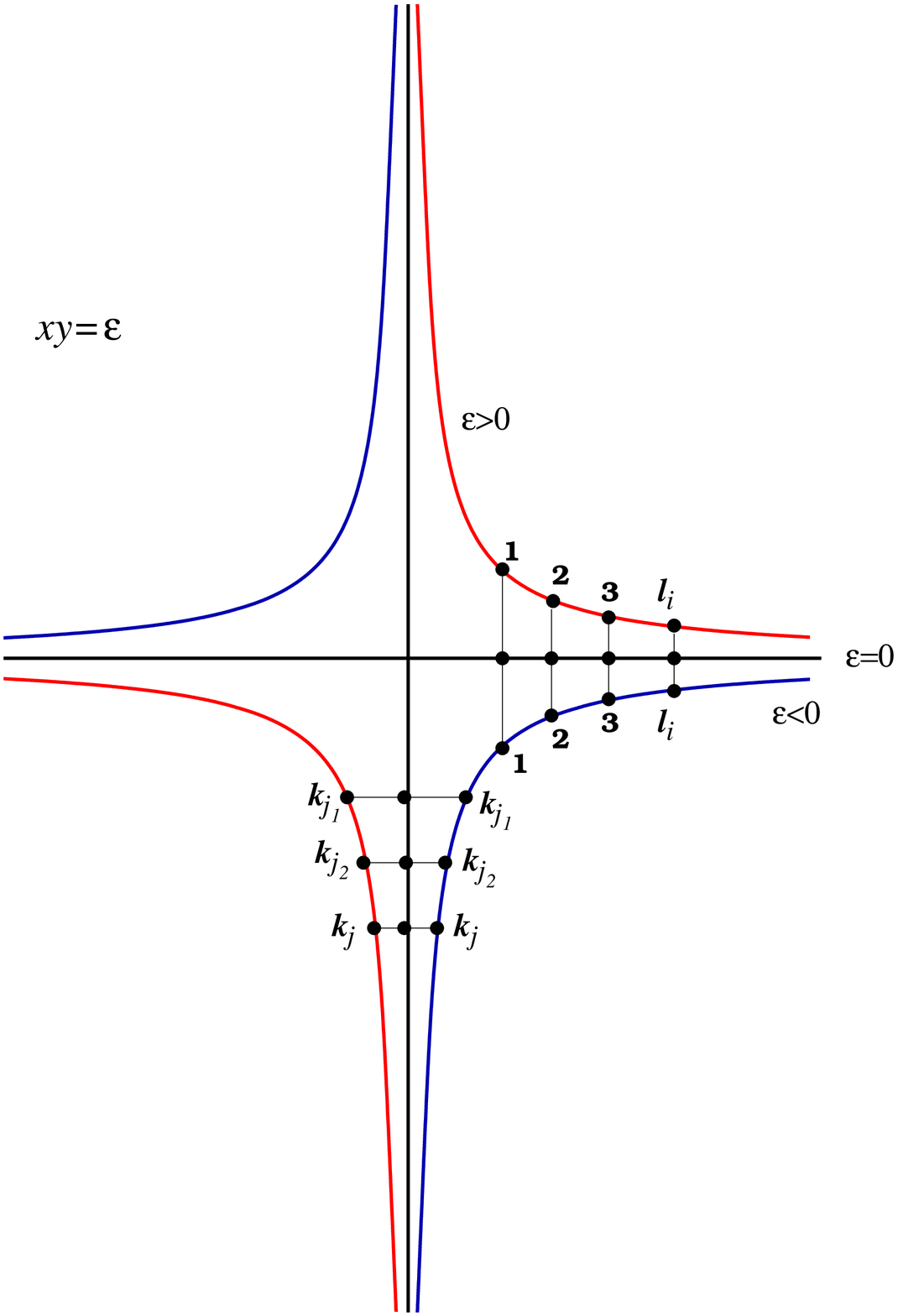,width=.5\linewidth}
\caption{$xy=\varepsilon$} \label{param}
\end{figure}
\begin{figure}[h]
\centering\epsfig{figure=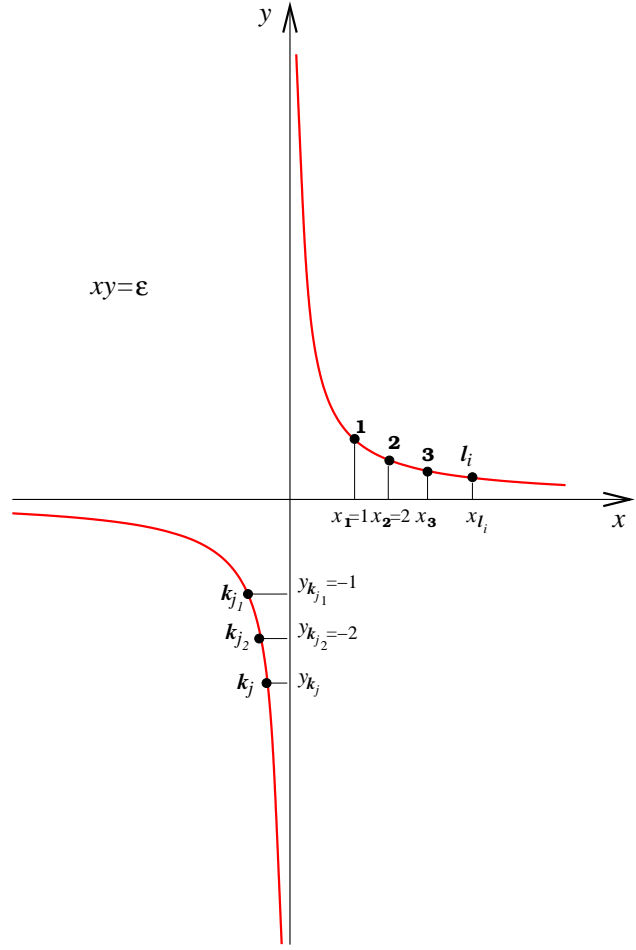,width=.5\linewidth}
\caption{$xy=\varepsilon$} \label{123}
\end{figure}

\begin{lemma}  Let $n\ge 6$ and all three points  1, 2, 3 are in one component of the boundary. Then corresponding cell is in the class $\Th(\overline{{\mathcal M}_{0,n}^\R})$  if and only if the component of the boundary which does not contain the points  1, 2,  3, contain odd number of the marked points. \label{lem1comp}
\end{lemma}
\begin{proof}
 Up to the rename of the marked points we can and we do assume that $l_{i_1},l_{i_2}\in \{1,2,3\}$. Applying Lemma \ref{lem123} we can set $l_{i_1}=1,l_{i_2}=2$. Then we denote $l_{i_3}=3$ for some $i_3$. Hence we obtain that abscissas of the points 1 and 2 are equal to 1 and 2, correspondingly. We denote the abscissa of the point 3  by $x_3$ and abscissas of the points $l_i\in \{l_1,\ldots,l_m\}\setminus \{1,2,3\}$ by $x_{l_i}$. Here $m\ge 3$ by the assumptions of the lemma asserting that the points 1,2 and 3 lie in one component.

Also we denote the ordinates of the points  $k_j\in\{ k_1,\ldots,k_{n-m}\}\setminus\{k_{1},k_{2}\}$ by $y_{k_j}$. 

Let us find the transition functions from the chosen coordinates $x,$  $y,$  $\ep$ to the coordinates  $z$ introduced in Definition~\ref{DefKoord}.

Point-wise the change of coordinates looks as follows:
$$
\begin{smallmatrix} i& 1&2&3& \ldots & l_i& \ldots&k_j& \ldots& k_{j_1}&k_{j_2} \\ z-\mbox{coordinates } &0&1&\infty& \ldots &z_{l_i}& \ldots&z_{k_j}& \ldots& z_{k_{j_1}} &z_{k_{j_2}} \\x-\mbox{coordinates  } &1&2&x_3 &\ldots& x_{l_i}& \ldots & \frac{\ep}{y_{k_j}}& \ldots & -\ep & -\frac{\ep}{2} \\ y-\mbox{coordinates  } &\ep&\frac{\ep}2&\frac{\ep}{x_3}& \ldots &\frac{\ep}{x_{l_i}}& \ldots&y_{k_j}& \ldots&-1&-2 \end{smallmatrix}$$ 

We are looking for the linear-fractional transformation  
$f(t)=\frac{at+b}{ct+d}$  which maps  $x$-coordinates of the points 1, 2 and 3  to their $z$-coordinates. Then $f(t)=(2-x_3)\frac{t-1}{t-x_3}$.

We compute the Jacobian of the change of the coordinates  $\ep, x_{l_i},y_{k_j}$  by the coordinates   $z_4,\ldots, z_n$, writing rows and columns of the matrix of partial derivatives in the following order. The rows  $z_{k_{1}}$, $z_{k_{2}} $, $z_{l_1}, \ldots, z_{l_m}$, $z_{k_1}, \ldots, z_{k_{n-m}}$ (without  $z_{l_{i_1}},\ z_{l_{i_2}},\ z_{l_{i_3}},\ z_{k_{j_1}}$, $z_{k_{j_2}} $). The columns: $\ep$, $x_3$, $x_{l_1}, \ldots, x_{l_m}$, $y_{k_1}, \ldots, y_{k_{n-m}}$ (without  $x_{l_{i_1}},\ x_{l_{i_2}},$ $y_{k_{j_1}}$, $y_{k_{j_2}} $, which are fixed, and $x_{l_{i_3}}$ which is in the second column). Note that $\frac{\partial z_{l_i}}{\partial x_{l_s}}=0$ if $s\ne i$, $\frac{\partial z_{l_i}}{\partial y_{k_t}}=0$ for all $i$ and $t$, $\frac{\partial z_{k_j}}{\partial x_{l_s}}=0$ for all $j$ and $s$. Also $\frac{\partial z_{k_j}}{\partial y_{k_t}}=0$ if  $t\ne j$. 

Then the Jacobi matrix of change of the coordinates has the following block-triangular form:
$$ F=\left(\begin{array}{ccc}F_1 & O_1 & O_2 \\ X_1 & D_1 & O_3\\ X_2 & O_4 & D_2 \end{array}\right),$$
where $$F_1=\left(\begin{array}{cc} \frac{\partial z_{k_{j_1}}}{\partial \ep} &  \frac{\partial z_{k_{j_1}}}{\partial x_3} \\  \frac{\partial z_{k_{j_2}}}{\partial \ep} &  \frac{\partial z_{k_{j_2}}}{\partial x_3}\end{array}\right),$$

$O_1, O_2, O_3, O_4$ are the zero matrices of the sizes $2\times (m-3)$, $2\times (n-m-2)$, $(m-3)\times (n-m-2)$ and $(n-m-2)\times (m-3)$, correspondingly, 

$X_1$ and $X_2$ are unknown matrices of sizes $ (m-3) \times 2$ and $(n-m-2) \times 2$, correspondingly,

$D_1$ is a diagonal  $ (m-3) \times  (m-3)$-matrix of partial derivatives $ \frac{\partial z_{l_{i}}}{\partial  x_{l_{i}}} $, and $D_2$ is a diagonal  $ (n-m-2) \times  (n-m-2)$-matrix  of partial derivatives $ \frac{\partial z_{k_{j}}}{\partial  y_{k_{j}}} $.

Then Jacobian $J=\det F=\det F_1 \cdot \det D_1\cdot \det D_2$. Let us compute each factor separately.

1. Determinant of the matrix $F_1$.
$$J_1=\det F_1=\det \left(\begin{array}{cc}  \frac{-(2-x_3)(1-x_3)}{(\ep+x_3)^2} & \frac{-(\ep+1)(\ep+2)}{(\ep+x_3)^2} \\ \frac{-(2-x_3)(2-2x_3)}{(\ep+2x_3)^2} & \frac{-(\ep+2)(\ep +4)}{(\ep+2x_3)^2}\end{array}\right) =$$ $$=
\frac{(2-x_3)(1-x_3)(\ep+2)}{(\ep+x_3)^2}{(\ep+2x_3)^2} \det\left( \begin{array}{cc}1&\ep+1 \\ 2& \ep+4\end{array}\right) = \frac{(2-x_3)(1-x_3)(4-\ep^2)}{(\ep+x_3)^2}{(\ep+2x_3)^2}.$$

2. Diagonal entries of  $D_1$ and $D_2$.

Since $z_{l_i}=(2-x_3)\cdot \frac{x_{l_i}-1}{x_{l_i}-x_3}$, we get $ \frac{\partial z_{l_{i}}}{\partial  x_{l_{i}}} =(2-x_3)(x_3-1)\cdot \frac{-1}{(x_{l_i}-x_3)^2}.$ Similarly, $$z_{k_j}= (2-x_3)\cdot \frac{\frac{\ep}{y_{k_j}}-1}{ \frac{\ep}{y_{k_j}}-x_3},$$
hence, $ \frac{\partial z_{k_{j}}}{\partial  y_{k_{j}}} =\ep(2-x_3)(x_3-1)\cdot \frac{1}{(\ep-y_{k_j}x_3)^2}.$ Determinant of the product of  $D_1$ and $D_2$ is the product of all found values.

Then 
\begin{equation} \label{eq:J>}  J= J_1 \det D_1\det D_2= (2-x_3)^{n-2}(x_3-1)^{n-2}(-1)^{m+1}\ep^{n-m-2}\cdot J_2,\end{equation}
where
$$J_2=\frac{4-\ep^2}{(\ep+x_3)^2(\ep+2x_3)^2}\prod_{i=4}^m \frac1{(x_{l_i}-x_3)^2} \prod_{j\in \{1,\ldots,n-m\}\setminus \{j_1,j_2\}}  \frac1{(\ep-y_{k_j}x_3)^2}.$$
Note that in the neighborhood of $\ep=0$ it holds that $J_2>0$, so this factor does not fluent on the sign of the Jacobian.

It follows from the formula  (\ref{eq:J>}) that if $(n-m)$ is even, then Jacobian $J$ does not change the sign while $\ep$ goes from $\ep<0$ to $\ep>0$. If $(n-m)$ is odd, then changes.

Let $(n-m)$ be even, i.e., there is an even number of marked points in the component of the boundary which does not contain the points 1, 2, 3. The cells of maximal dimension corresponding to  $\ep>0$ and $\ep<0$ have opposite orientations. Hence  arguing as in the proof of Lemma~\ref{K12345} we obtain that  in the coordinates $(z_4,\ldots, z_n) $ these cells have opposite orientations. Therefore, their common boundary is included twice into the expression for $\Th(\overline{{\mathcal M}_{0,n}^\R})$ and with the opposite signs, hence its influence is zero. 

Let now  $(n-m)$ be odd,  i.e., there is an odd number of marked points in the component of the boundary which does not contain the points 1, 2, 3. The cells of maximal dimension corresponding to  $\ep>0$ and $\ep<0$ have opposite orientations. Hence  we obtain that  in the coordinates $(z_4,\ldots, z_n) $ these cells have the same orientations, since the sign of $J$ is changed in the point $\ep=0$. Therefore, their common boundary is included twice with the same sign into the expression for $\Th(\overline{{\mathcal M}_{0,n}^\R})$. After dividing by 2 and taking the result modulus 2 we get that the common boundary of these two cells is in the class  $\Th(\overline{{\mathcal M}_{0,n}^\R})$. This concludes the proof of the lemma.

\end{proof}

\subsubsection{The case when the points  1, 2, 3 are in the different components of the boundary}
The case under consideration is presented at Figure 14.
\begin{figure}[h]
\centering\epsfig{figure=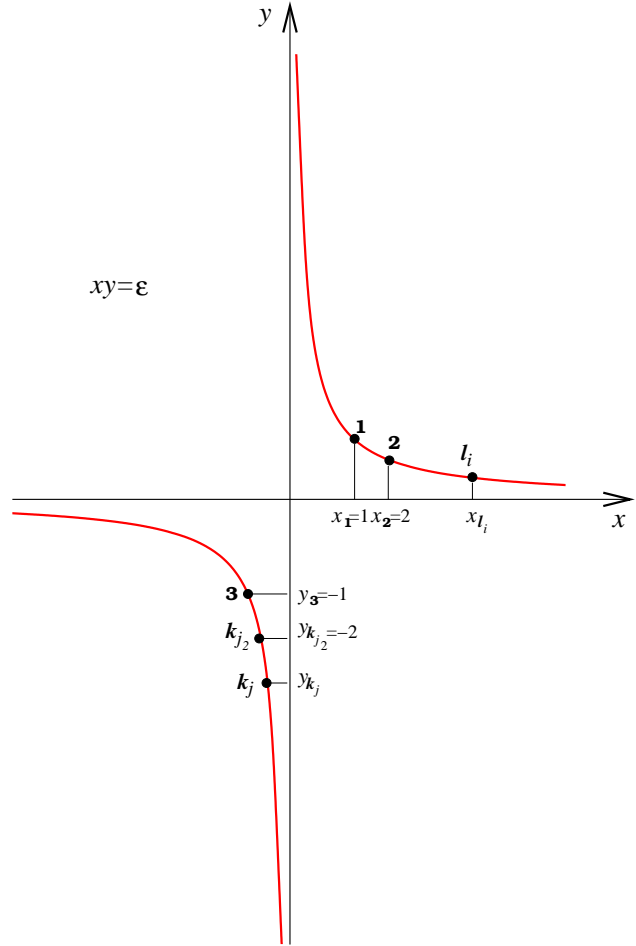,width=.5\linewidth}
\caption{$xy=\varepsilon$} \label{12}
\end{figure}

\begin{lemma} Let $n\ge 6$ and only two of the points 1, 2, 3 are in the same component of the boundary $K$. Then $K$ is in the class   $\Th(\overline{{\mathcal M}_{0,n}^\R})$ if and only if there are odd number of marked points on the component of $K$ which contains the third point from the set $\{1,2,3\}$.
\end{lemma}
\begin{proof}
Applying Lemma \ref{lem123}, up to the rename of the marked points we can and we do set  $l_{i_1}=1,l_{i_2}=2$, and $k_{j_1}=3$. Then we obtain that the abscissas of the points 1 and 2 are 1 and 2, correspondingly, and the ordinate of the point  3 is $-1$. We denote the abscissas of the points $l_i\in \{l_1,\ldots,l_m\}\setminus \{1,2\}$ by $x_{l_i}$, here $m\ge 2$.
Let us denote ordinates of the points $k_j\in\{ k_1,\ldots,k_{n-m}\}\setminus\{k_{j_1},k_{j_2}\}$ by $y_{k_j}$. 

We find the transition functions from the chosen coordinates  $x$, $y$, $\ep$ to the coordinates  $z$, introduced in Definition~\ref{DefKoord}.

Point-wise the change of coordinates looks as follows:
$$
\begin{smallmatrix} i& 1&2&3& \ldots & l_i& \ldots&k_j& \ldots& k_{j_2} \\ z-\mbox{coordinates } &0&1&\infty& \ldots &z_{l_i}& \ldots&z_{k_j}& \ldots &z_{k_{j_2}} \\ x-\mbox{coordinates } &1&2&-\ep &\ldots& x_{l_i}& \ldots & \frac{\ep}{y_{k_j}}& \ldots  & -\frac{\ep}{2} \\ y-\mbox{coordinates } &\ep&\frac{\ep}2& -1& \ldots &\frac{\ep}{x_{l_i}}& \ldots&y_{k_j}& \ldots&-2 \end{smallmatrix}$$ 

We are looking for the linear-fractional transformation  
$f(t)=\frac{at+b}{ct+d}$  which maps  $x,y,\ep$-coordinates of the points 1, 2 and 3 to their $z$-coordinates. Then $f(t)=(2+\ep)\frac{t-1}{t+\ep}$.

We compute the Jacobian of the change of the coordinates $\ep, x_{l_i},y_{k_j}$ by the coordinates   $z_4,\ldots, z_n$, writing rows and columns of the matrix of partial derivatives in the following order.
The rows: $z_{k_{j_2}} $, $z_{l_1}, \ldots, z_{l_m}$, $z_{k_1}, \ldots, z_{k_{n-m}}$ (without  $z_{l_{i_1}},\ z_{l_{i_2}},\ z_{k_{j_1}}$, $z_{k_{j_2}} $). The columns: $\ep$, $x_{l_1}, \ldots, x_{l_m}$, $y_{k_1}, \ldots, y_{k_{n-m}}$ (without  $x_{l_{i_1}},\ x_{l_{i_2}},\ y_{k_{j_1}}$, $y_{k_{j_2}} $, which are fixed). As in the previous lemma we note that $\frac{\partial z_{l_i}}{\partial x_{l_s}}=0$ if $s\ne i$, $\frac{\partial z_{l_i}}{\partial y_{k_t}}=0$ for all $i$ and $t$, $\frac{\partial z_{k_j}}{\partial x_{l_s}}=0$  for all $j$ and $s$, also $\frac{\partial z_{k_j}}{\partial y_{k_t}}=0$  if $t\ne j$.  Then the Jacobi matrix of this change of the coordinates is triangular with zeros above the diagonal:
$$ F=\left(\begin{array}{cc}\frac{\partial z_{k_{j_2}}}{\partial \ep}& 0 \\ * & D  \end{array}\right),$$
where $*$ denotes unknown  elements of the first column, the first row is zero, except the entry in the position  $(1,1)$, and $D$ is a diagonal   $ (n-4) \times  (n-4)$ matrix. First $(m-2) $ diagonal entries of $D$ have the form $ \frac{\partial z_{l_{i}}}{\partial  x_{l_{i}}} $, next  $ (n-m-2) $ entries have the form $ \frac{\partial z_{k_{j}}}{\partial  y_{k_{j}}} $.

Let us compute all the factors separately:

1. $ \frac{\partial z_{k_{j_2}}}{\partial \ep}= \frac{4-\ep^2}{\ep^2}$.

2.  $ \frac{\partial z_{l_{i}}}{\partial  x_{l_{i}}} = \frac{(2+\ep)(1+\ep)}{(x_{l_i}+\ep)^2}. $

3. $ \frac{\partial z_{k_{j}}}{\partial  y_{k_{j}}} =-\frac{(2+\ep)(1+\ep)}{\ep(y_{k_j}+1)^2}. $

Then 
\begin{equation} \label{eq:J}  J= (-1)^{n-m-1} \frac1{\ep^{n-m}} (2+\ep)^{n-4} (1+\ep)^{n-4} \prod_{i=3}^m \frac1{(x_{l_i}- \ep)^2} \cdot \prod_{j\in \{1,\ldots,n-m\}\setminus \{3,j_2\}}  \frac1{(1+y_{k_j})^2}.\end{equation}
Note that in the neighborhood of $\ep=0$  the sign of the Jacobian is determined by the factor $1/\ep^{n-m}$ only.

Formula (\ref{eq:J}) implies that if $(n-m)$ is even, then the Jacobian $J$ does not change the sign going from  $\ep<0$ to $\ep>0$, and if $(n-m)$ is odd then $J$ changes the sign. 

Thus, repeating the arguments at the end of Lemma \ref{lem1comp} we obtain that if  $(n-m)$ is even $K$ is not in the class  $\Th(\overline{{\mathcal M}_{0,n}^\R})$, and if  $(n-m)$ is odd, then $K$ is in this class, as required. 

\end{proof}

Consequent application of these lemmas concludes the proof of Theorem \ref{them1}.

\section*{Acknowledgments}

The authors would like to express their deep gratitude to Sergei M. Natanzon for paying their attention to the problems related to real moduli spaces of algebraic curves, and for his fruitful comments and discussions.


(N. Ya. Amburg)

 ITEP, Moscow, 117218, Russia

 National Research University Higher School of Economics, 
International Laboratory of Representation 
Theory and Mathematical Physics,
20 Myasnitskaya Ulitsa, Moscow 101000, Russia 

\bigskip

(E. M. Kreines)

 Lomonosov Moscow State University, Moscow, Russia

ITEP, Moscow, 117218, Russia

\end{document}